\documentclass[12pt]{amsart}

\usepackage[english]{babel}
\usepackage{enumitem}
\usepackage[ansinew]{inputenc}
\usepackage{amsmath,amsthm}
\usepackage{hyperref}
\usepackage[T1]{fontenc}
\usepackage{lmodern}
\usepackage{amssymb}
\usepackage{amscd}
\usepackage{marginnote}

\setitemize{leftmargin=*}
\setenumerate[1]{label=(\alph*)}
\setenumerate[2]{label=(\roman*)}

\theoremstyle{plain}
\newtheorem{thm}{Theorem}
\newtheorem{lemma}[thm]{Lemma}
\newtheorem{korollar}[thm]{Corollary}
\newtheorem{prop}[thm]{Proposition}

\theoremstyle{remark}

\theoremstyle{definition}
\newtheorem{example}[thm]{Example}
\newtheorem{defin}[thm]{Definition}

\newcommand{\spec}[1]{\ensuremath{\text{Spec }#1}}

\newcommand{\mymin}[1]{\min\left(#1\right)}
\newcommand{\pl}[0]{\mathbb{P}^1_k}
\newcommand{\al}[0]{\mathbb{A}^1_k}
\newcommand{\ku}[0]{k^\times}
\newcommand{\chark}[1]{\ensuremath{\text{char }#1}}
\newcommand{\kTf}[0]{k[T,T^{-1}]}

\title[Reduction behavior for elliptic curves over $\mathbb{P}^1$]{Minimal number of points with bad reduction for elliptic curves over $\mathbb{P}^1$}

\author{Johannes Sprang}
\email{johannessprang@gmx.de}
\date{}

\hypersetup{
pdfauthor = {J. Sprang},
pdftitle = {Minimal number of points with bad reduction for elliptic curves over P1}
}

\subjclass[2000]{14H52; 11R58}

\begin{document}

\begin{abstract}
In this work we use elementary methods to discuss the question of the minimal number of points with bad reduction over $\pl$ for elliptic curves $E/k(T)$ which are non-constant resp. have non-constant $j$-invariant.
\end{abstract}

\maketitle

\section[Introduction]{Introduction}

	It is a well known fact going back to J. Tate (see \cite[ch.5, \S 8]{hus}) that there are no elliptic curves over $\mathbb{Q}$ with good reduction everywhere. This was generalized by J.-M. Fontaine  \cite{fon} 
	to abelian varieties over $\mathbb{Q}$. In \cite{schoof} 
	R. Schoof answers the question about the existence of non-zero abelian varieties over $\mathbb{Q}$ with bad reduction at just one prime.\par
	In the case of function fields there are trivial examples of elliptic curves with good reduction everywhere coming from elliptic curves over the field of constants. Thus it is natural to impose further conditions to exclude those trivial examples.\par 
	A. Beauville proves in \cite{bea} 
	that every semi-stable elliptic surface $S\rightarrow \mathbb{P}^1_\mathbb{C}$ has at least four singular fibres. He also showes in this paper that every non-isotrivial elliptic surface $S\rightarrow \mathbb{P}^1_k$ for $k$ an algebraic closed field of characteristic $p>3$ has at least three singular fibres by applying the Riemann-Hurwitz-Theorem to the induced $j$-map. In a remark he claims that similar methods may be used to prove that there are no elliptic curves over $\mathbb{A}^1_k$ with non-constant $j$-invariant and good reduction everywhere in characteristics $2$ and $3$.\par
	In this paper we take a more elementary approach to determine the minimal number of points with bad reduction over $\pl$ for elliptic curves $E$ over $k(T)$, if we require $E$ to be non-constant resp. $j(E)\notin k$. In characteristic different from $2$ and $3$ a non-constant elliptic curve has at least $2$ points with bad reduction. If we furthermore demand $j(E)\notin k$, then there are at least $3$ points with bad reduction. \par
In characteristics $2$ and $3$ we show that every non-constant elliptic curve has at least $1$ point with bad reduction. If we further require $E$ to have non-constant $j$-invariant i.e. $j(E)\notin k$, then we obtain at least $2$ points with bad reduction.
By giving examples we will show that the obtained bounds are sharp.\par

I would like to thank N. Naumann for suggesting this work to me and for his mentoring advice as well as R. Schoof for his initiating idea towards an elementary solution of the problem.

\section[Notation and conventions]{Notation and conventions}

In this section $k$ denotes an algebraically closed field. Let $E$ be an elliptic curve over a function field $K/k(T)$. We call $E$ \textit{constant}, if there exists a Weierstrass equation for $E$ with coefficients in $k$. We say $E$ has\textit{ constant $j$-invariant}, if $j(E)\in k$.
Clearly every constant elliptic curve has constant $j$-invariant. Conversely every elliptic curve with constant $j$-invariant gets constant over some finite extension of $K$. For the following we refer to \cite{liu}[ch. 10]. Let $S=\spec{A}$ be an affine Dedekind scheme (of dimension one) with function field $K(S)$ and $E/K(S)$ an elliptic curve. We say $E$ has good reduction at $s\in S$, if $E$ admits a smooth model over $\spec{\mathcal{O}_{S,s}}$.\par
  In the following we will consider affine open subsets $S=\spec{A}\subseteq \pl$ and elliptic curves $E$ over $K(S)=k(T)$. Since $\text{Pic}(A)=0$ we can choose a globally minimal Weierstrass equation $W$ for $E$ over $S$. Then $E$ has good reduction at $s\in S$, if and only if $\nu_s(\Delta_W)=0$. Here $\nu_s$ denotes the valuation associated to $\mathcal{O}_{S,s}$.

\section[Statement of the results]{Statement of the results}

\begin{thm}\label{statement_1}
	Let $k$ be an algebraically closed field with characteristic different from $2$ and $3$. Further let $E/k(T)$ be an elliptic curve. Then
	\begin{enumerate}[leftmargin=*]
		\item there are at least three points with bad reduction over $\pl$, if $E$ has non-constant $j$-invariant
		\item and at least two points with bad reduction, if $E$ is non-constant.
	\end{enumerate}
	The bounds given in (a) and (b) are sharp.
\end{thm}

\begin{thm}\label{statement_2}
	Let $k$ be an algebraically closed field of characteristic $2$ or $3$ and let $E /k(T)$ be an elliptic curve. Then
	\begin{enumerate}[leftmargin=*]
		\item there are at least two points with bad reduction over $\pl$, if $E$ has non-constant $j$-invariant.
		\item and at least one point with bad reduction, if $E$ is assumed to be non-constant.
	\end{enumerate}
	The bounds given in (a) and (b) are sharp.
\end{thm}

\section[Proofs]{Proofs}
\subsection[Proofs for characteristic different from $2$ and $3$]{Proofs for characteristic different from 2 and 3}
In this section $k$ denotes an algebraically closed field. We prove the results in characteristic different from $2$ and $3$ by applying Mason's $abc$-inequality to the discriminant of a globally minimal Weierstrass equation:

\begin{defin}
For $x\in k(T)$ we call
\begin{equation}\label{eqd1_1} 
	H(x):=\sum_{\nu} -\min(0,\nu(x))
\end{equation}
where the sum is over all canonical valuations on $k(T)$, the \underline{height} of $x$. For more detailed background we refer to \cite[ch. I,VI]{mas}.
\end{defin}

\begin{lemma}[Mason's ABC-inequality]\label{mason}
	Let $\mathfrak{V}$ be a finite set of canonical valuations $\nu$ on $k(T)$.
	Let $\gamma_1,\gamma_2,\gamma_3 \in k(T)^\times$ with $\gamma_1+\gamma_2+\gamma_3=0$ and such that
	$\nu(\gamma_1)=\nu(\gamma_2)=\nu(\gamma_3)$ for each valuation $\nu$ not in $\mathfrak{V}$.
		\begin{enumerate}[leftmargin=*]
			\item If $\text{char} (k)=0$, we have either $\frac{\gamma_1}{\gamma_2}\in k$ or $$H\left(\frac{\gamma_1}{\gamma_2}\right)\leq \left|\mathfrak{V}\right|-2$$
			\item If $\text{char} (k)=p>0$, we have either $\frac{\gamma_1}{\gamma_2}\in k(T)^p$ or $$H\left(\frac{\gamma_1}{\gamma_2}\right)\leq \left|\mathfrak{V}\right|-2$$
		\end{enumerate}
	Here $H$ denotes the height function on $k(T)$. Compare \cite{mas}[ch. I,VI].
\end{lemma}
\begin{proof}
	see \cite[P.14 Lemma 2; P.97 Lemma 10]{mas}
\end{proof}

\begin{lemma}\label{lemma2}
Let $k$ be an algebraically closed field of characteristic different from $2$ and $3$.
For all $A,B\in \kTf$ with $A^3-B^2\in \kTf^\times$ we have $A,B \in \kTf^\times \cup \{0\}$.
\end{lemma}
\begin{proof}

	In the following we may assume $AB\neq0$ since otherwise we are done.
	Let $A,B \in \kTf$ such that $A^3-B^2=aT^l$ for $a\in k^\times,l\in\mathbb{Z}$.
	For the polynomials 
	\begin{equation*}
		q_A:=\frac{A}{T^{\nu_T(A)}}, \quad q_B:=\frac{B}{T^{\nu_T(B)}}
	\end{equation*}
	let us define:
	\begin{equation*}
		\begin{split}
			\mathfrak{V}_A:= & \left\{\nu_g : g\in \text{support}(q_A)  \right\}\\
			\mathfrak{V}_B:= & \left\{\nu_g : g\in \text{support}(q_B) \right\}\\
		\end{split}
	\end{equation*}
	Since $A$ and $B$ are coprime in $\kTf$ we see
 	\begin{equation}\label{eq2_4}
		\mathfrak{V}_A \cap \mathfrak{V}_B=\emptyset
	\end{equation}
	and for $\mathfrak{V}:= \mathfrak{V}_A \cup \mathfrak{V}_B \cup  \left\{\nu_\infty,\nu_T\right\}$ we have:
	\begin{equation}\label{eq2_1}
		\left|\mathfrak{V}\right| \leq \deg q_A + \deg q_B + 2
	\end{equation}
	The definition of the height function gives the following estimates:

	\begin{align}
		H\left(\frac{A^3}{B^2}\right)& \geq \sum_{\nu \in \mathfrak{V}_B} 
		 -\mymin{0,\nu\left(A^3\right)-\nu\left(B^2\right)}= 2 \deg q_B \label{eq2_2B} \\
		H\left(\frac{B^2}{A^3}\right)&\geq \sum_{\nu \in \mathfrak{V}_A} 
		 -\mymin{0,\nu\left(B^2\right)-\nu\left(A^3\right)}=3 \deg q_A \label{eq2_2A}
	\end{align}
	\par
	\textit{(1)} Now we first suppose \emph{[$\text{char}(k)=0$ and $\frac{A^3}{B^2},\frac{B^2}{A^3}\notin k$]} or \emph{[$\text{char}(k)=p>0$ and $\frac{A^3}{B^2},\frac{B^2}{A^3}\notin k(T)^p$]}, so that we may apply Mason's inequality:
	\begin{equation}\label{eq2_5AB}
 			H\left(\frac{B^2}{A^3}\right)\leq	 \left|\mathfrak{V}\right|-2, \quad H\left(\frac{A^3}{B^2}\right)\leq	 \left|\mathfrak{V}\right|-2
 	\end{equation}
	We get
	\begin{equation*}
	3\deg q_A + 2 \deg q_B  \stackrel{\eqref{eq2_2B},\eqref{eq2_2A},\eqref{eq2_5AB}}{\leq}	2(\left|\mathfrak{V}\right|-2) \stackrel{\eqref{eq2_1}}{\leq} 2\deg q_A +2\deg q_B 
	\end{equation*}
	as well as:
	\begin{equation*}
	2 \deg q_B  \stackrel{\eqref{eq2_2B}}{\leq} H\left(\frac{A^3}{B^2}\right)
	\stackrel{\eqref{eq2_5AB}}{\leq}
	 (\left|\mathfrak{V}\right|-2) \stackrel{\eqref{eq2_1}}{\leq} \deg q_A +\deg q_B
	\end{equation*}
	Together this yields $\deg q_B \leq \deg q_A\leq0$ and from the definition of $q_A$ and $q_B$ we get $A,B\in \kTf^\times$ as desired.\par
	\textit{(2)} If we have \emph{[$\text{char}(k)=0$ and $\frac{A^3}{B^2},\frac{B^2}{A^3}\in k$]}, we get $\frac{q_A^3}{q_B^2}\in k$. And since $q_A$ and $q_B$ are coprime as polynomials we have $q_A,q_B\in k$.\par
	\textit{(3)} The remaining case \emph{[$\text{char}(k)=p>0$ and $\frac{A^3}{B^2},\frac{B^2}{A^3}\in k(T)^p$]} may be reduced to the first case.
	Let $s:=\sup \left(s: \frac{A^3}{B^2}\in k(T)^{p^s}\right)$. We may assume $s<\infty$ since otherwise we would have $\frac{A^3}{B^2}\in k$ and therefore $A^3,B^2 \in \kTf^\times$.
	Since we have $\frac{A^3}{B^2}\in k(T)^{p^s}$ and since $p$ is different from $2$ and $3$ there are $\tilde{q_A},\tilde{q_B}$ and some $n\geq0$ such that:
	\begin{equation*}
	  \tilde{q_A}^{p^s}=q_A,\text{ }\tilde{q_B}^{p^s}=q_B \text{ and } 3\nu_T(A)-2\nu_T(B)=np^s
	\end{equation*}
	After dividing $A^3-B^2=aT^l$ by $T^{2\nu_T(B)}$ we have:
	\begin{equation*}
	  (\tilde{q_A}^3 T^{n}-\tilde{q_B}^2)^{p^s} = aT^{l-2\nu_T(B)}
	\end{equation*}
	Multiplying $\tilde{q_A}^3 T^{n}-\tilde{q_B}^2\in \kTf^\times$ by an appropriate $T^m$ such that $3|(m+n)$ and $2|m$
	yields:
	\begin{equation*}
	 (\tilde{q_A}T^{\frac{n+m}{3}})^3 -(\tilde{q_B}T^{\frac{m}{2}})^2\in \kTf^\times
	\end{equation*}
	By the choice of $s$ we have $\frac{\tilde{q_A}^3 T^{n+m}}{\tilde{q_B}^2 T^m}=\frac{\tilde{q_A}^3 T^{n}}{\tilde{q_B}^2}\notin k(T)^p$ and thus the first case applies.
\end{proof}

\begin{korollar}\label{korollar1}
	For $l\in\mathbb{Z}$ and $u\in k^\times$ consider $A,B\in \kTf$ with $A^3-B^2=uT^l$.
		\begin{enumerate}[leftmargin=*]
			\item If $AB\neq 0$, then $(A,B)\in\left\{ (a T^{2n},b T^{3n}):a,b\in k^\times n\in\mathbb{Z} \right\}$
			\item If $A=0$, then $(A,B)\in\left\{ (0,b T^{n}):b\in k^\times, n\in\mathbb{Z}\right\}$
			\item If $B=0$, then $(A,B)\in\left\{ (a T^{n},0):a\in k^\times ,n\in\mathbb{Z}\right\}$
		\end{enumerate}	
\end{korollar}
\begin{proof}
	By Lemma \ref{lemma2} we have $A,B\in \kTf^\times\cup \{0\}$, say $A=a T^{l_A},B=b T^{l_B}$ with $a,b\in k$, $l_a,l_B\in\mathbb{Z}$. Thus we have an equation of the form:
	$$a^3 T^{3l_A}-b^2 T^{2l_B}=u T^l$$
	Now the corollary follows from the fact that $\{T^i\}_{i\in \mathbb{Z}}$ is a $k$-basis for $\kTf$.
\end{proof}

\begin{korollar}\label{korollar2}
	Let $A,B\in k[T]$ satisfying $A^3-B^2\in k^\times$. Then we have $A,B\in k$.
\end{korollar}
\begin{proof}
	Setting $l=0$ in Corollary \ref{korollar1} we see that the result even holds for $A,B\in \kTf$
\end{proof}

Using the two corollaries above we can prove the main results in all characteristics different from $2$ and $3$.

\begin{proof}[Proof of Theorem \ref{statement_1} (a)]
	We show that every elliptic curve $E/k(T)$ with bad reduction restricted to two points has $j(E)\in k$. 
Without loss of generality we may assume these points to be $\{0,\infty\}$, so let $E/k(T)$ be an elliptic curve with good reduction everywhere over $\spec{\kTf}$.
 Since we are in characteristic different from $2$ and $3$, we can choose for $E$ a globally minimal Weierstrass equation over $\spec{\kTf}$ of the form
		\begin{equation*}
			W:y^2=x^3-3Ax+2B
		\end{equation*}
		with $A,B\in k[T,T^{-1}]$. Since $E$ has good reduction everywhere over $\spec{\kTf}$, we have $\Delta_W=1728(A^3-B^2) \in k[T,T^{-1}]^\times$. But this is just possible for $(A,B)$ as in Corollary \ref{korollar1}. It is easy to see that $E$ has constant $j$-invariant by inserting the results from Corollary \ref{korollar1} into the formula $j(E)=12^3\frac{A^3}{A^3-B^2}$. \par
		If we remove three points we can just write down an elliptic curve with non-constant $j$-invariant and good reduction everywhere:
\end{proof}
		\begin{example}\label{proof_1ea}
			The elliptic curve $E/k(T)$ defined by the Weierstrass equation in Legendre form:
			\begin{equation*}
				W:\quad y^2=x(x-1)(x-T)
			\end{equation*}
			has $j$-invariant $j(E)=2^8\frac{(T^2-T+1)^3}{T^2(T-1)^2}$, so $E$ has non-constant $j$-invariant.
			From the discriminant $\Delta_W=16T^2(T-1)^2$ we see that $W$ is a globally minimal Weierstrass equation for $E$ over $\mathbb{P}^1_k\setminus \{ 0,1,\infty \}$ and that it has good reduction outside $\{ 0,1,\infty \}$.
		\end{example}		
	
\begin{proof}[Proof of Theorem \ref{statement_1} (b)]
		Let $E$ be an elliptic curve over $k(T)$. Since we are in characteristic different from $2$ and $3$ we can choose for $E$ a globally minimal Weierstrass equation over $\al=\spec{k[T]}$ of the form:
		\begin{equation*}
			W:y^2=x^3-3Ax+2B
		\end{equation*}
		with $A,B\in k[T]$. If $E$ has good reduction everywhere over $\mathbb{A}^1_k$, we have $\Delta_W=1728(A^3-B^2)\in k^\times$, but due to Corollary \ref{korollar2} this is just possible for $A,B$ constant. Thus $E$ is constant. The following example shows that there are non-constant elliptic curves, if we allow bad reduction in one more point.
\end{proof}
		
		\begin{example}\label{proof_1eb}
			The elliptic curve $E$ over $k(T)$ defined by
			\begin{equation*}
				W:\quad y^2=x^3-3 T^2 x
			\end{equation*}
			has $j$-invariant $j=1728$ and discriminant $\Delta_W=1728 T^6$. Thus $W$ is a globally minimal Weierstrass equation for $E$ over $\spec{\kTf}$ and has good reduction everywhere over $\pl\setminus \{0,\infty\}$ since $\Delta_W \in \kTf^{\times}$. It is easy to see that $E$ is non-constant: Suppose we had a Weierstrass equation
			\begin{equation*}
				y^2+a_1 xy+a_3 y=x^3+ a_2 x^2+a_4 x+a_6
			\end{equation*}
			with constant coefficients. Then a suitable substitution would give us a constant two-parameter Weierstrass equation
			\begin{equation*}
				W':\quad y'^2=x'^3+ ax'+b
			\end{equation*}
			with constant coefficients $a,b\in k$. Comparing the Weierstrass equations $W$ and $W'$ yields $-3 T^2= a u^4$ for some $u\in k(T)$, which is impossible.
		\end{example}

\subsection[Proofs for characteristics $2$ and $3$]{Proofs for characteristics 2 and 3}
In characteristics $2$ and $3$ the non-existence of elliptic curves $E$ with good reduction everywhere over $\al$ and non-constant $j$-invariant can be obtained by applying the following Lemma to a globally minimal Weierstrass equation of $E$.

\begin{lemma}\label{proof_23_ldisc}
	Let $k$ be an algebraically closed field of characteristic $p>0$.
	Let 
	\begin{equation}\label{eql2_1}
		f(Y)=Y^{p^r}-Y^n Q_1^m - Q_1^{m+1} Q -c \in \left(k[T]\right) [Y]
	\end{equation}
	with $Q_1,Q\in k[T]$, $Q_1$ non-constant, $c\in k^\times$, $r,m,n\geq 1$ integers with $m<p^r$ such that $m$ and $n$ are coprime to $p$. Then $f(Y)$ has no roots in $k[T]$.
\end{lemma}
\begin{proof}
	Define 
		\begin{eqnarray*}
			w :\quad k[T] & \rightarrow & \mathbb{Z} \cup \{\infty\} \\
			g & \mapsto & \sup\left\{s:g\in k[T]^{p^s}\right\} 
		\end{eqnarray*}
		(for non constant $g\in k[T]$ we have $w(g)<\infty$). We now prove the Lemma by induction over $w(Q_1)$.\par
		 As first case we consider \underline{$s:=w(Q_1)<r$:}
			Let $f$ be an arbitrary polynomial of the form \eqref{eql2_1} as given in the Lemma with $w(Q_1)<r$. Let us further assume we had a solution $A\in k[T]$ to $f(Y)=0$:
			\begin{equation*}\label{eql2_2}
				A^{p^r}-c=Q_1^m \left( A^n+Q_1 Q \right) \text{ in } k[T]
			\end{equation*}
			From this equation we can see that $Q_1$ is coprime to $A$ and thus $Q_1$ is also coprime to $A^n+Q_1 Q$. Since $A^{p^r}-c \in k[T]^{p^r}$ we have $Q_1^m \left( A^n+Q_1 Q \right)  \in k[T]^{p^r}$. Since $k$ is algebraically closed we obtain from $(Q_1,A^n+Q_1 Q)=1$, that $Q_1^m\in k[T]^{p^r}$. And since $m$ is coprime to $p$ we get $Q_1\in k[T]^{p^r}$ contradicting $w(Q_1)<r$.\par
			\underline{$s-1\rightarrow s$:}
			We may assume $s\geq r$ since $s<r$ has already been treated.
			Suppose we have proven the Lemma for all polynomials $Q_1$ with $w(Q_1)<s$. Now consider a polynomial $f$ of the form \eqref{eql2_1} as in the Lemma with $w(Q_1)=s$. Further assume we had an $A\in k[T]$ with $f(A)=0$:
			\begin{equation*}\label{eql2_3}
				A^{p^r}-c=Q_1^m \left( A^n + Q_1 Q\right)
			\end{equation*}
			Again since $(Q_1,A^n+Q_1 Q)=1$ and $(m,p)=1$ we have $Q_1 \in k[T]^{p^r}$. Let $\tilde{Q}_1\in k[T]$ with $\tilde{Q}_1^{p^r}=Q_1$, further let $\tilde{c}\in k$ such that $\tilde{c}^{p^r}=c$. Thus we have:
			\begin{equation}\label{eql2_4}
				\left( \frac{A-\tilde{c}}{\tilde{Q}_1^m} \right)^{p^r}=A^n+Q_1 Q
			\end{equation}
			Setting $g:=\frac{A-\tilde{c}}{\tilde{Q}_1^m}$ and inserting $A=\tilde{c}+g\tilde{Q}_1^m$ in \eqref{eql2_4} yields:
			\begin{align*}\label{eql2_5}
				g^{p^r}&=(\tilde{c}+g\tilde{Q}_1^m)^n+Q_1 Q\\
							&=\tilde{c}^n+n \tilde{c}^{n-1} g\tilde{Q}_1^m+ \underbrace{\sum_{i=2}^{n}\binom{n}{i}\tilde{c}^{n-i} g^i \tilde{Q}_1^{im}+\tilde{Q}_1^{p^r} Q}_{=:\tilde{Q}_1^{m+1}\tilde{Q} \text{ using } p^r>m}
			\end{align*}
			
			Thus we have a solution $g\in k[T]$ of the equation
			\begin{equation*}\label{eql2_7}
				0=Y^{p^r}-\underbrace{n \tilde{c}^{n-1}}_{=:c'} Y\tilde{Q}_1^m-\tilde{Q}_1^{m+1}\tilde{Q}-\tilde{c}^n
			\end{equation*}			
			which is of the form \eqref{eql2_1} as considered in the Lemma (the constant $c'$ doesn't matter since we can pull it into $\tilde{Q}_1$).
			And because of $w(\tilde{Q}_1)=w(Q_1)-r<w(Q_1)$ this solution contradicts the induction hypothesis.
\end{proof}

This Lemma may be used to prove Part (a) of Theorem \ref{statement_2}, i.e. every elliptic curve $E/k(T)$ with non-constant $j$-invariant has at least two points with bad reduction.

\begin{proof}[Proof of Theorem \ref{statement_2} (a)]
 	It suffices to show that there is no elliptic curve with non-constant $j$-invariant and with good reduction everywhere over $\mathbb{A}^1_k$.
	Suppose we had such an elliptic curve $E$.
	We first prove this for characteristic $2$.
	Let
	\begin{equation*}
		W:	y^2+a_1 xy+a_3 y=x^3+ a_2 x^2+a_4 x+a_6
	\end{equation*}
	$a_i \in k[T]$ be a globally minimal Weierstrass equation for $E$ over $\mathbb{A}^1_k$. We have:
	\begin{equation}\label{proof_2p1_1}
		\Delta_W=a_3^4+a_1^3 a_3^3 +a_1^4 \left(a_1^2a_6 + a_1 a_3 a_4 + a_2 a_3^2 +a_4^2 \right), \quad  j(E)=\frac{a_1^{12}}{\Delta_W}
	\end{equation}
	By our assumptions on $E$ we must have $\Delta_W \in k^\times$ and $j(E)\notin k$. Then  $j(E)\notin k$ implies $a_1\notin k$. But by \eqref{proof_2p1_1} $a_3 \in k[T]$ would be a solution to
	\begin{equation*}
		Y^4+a_1^3Y^3+a_1^4(a_1^2a_6+a_1 a_3 a_4+a_2 a_3^2+a_4^2)-\Delta_W=0
	\end{equation*}
	contradicting Lemma \ref{proof_23_ldisc}  since $a_1\notin k$.\par		
	In characteristic $3$ we can choose a globally minimal Weierstrass equation for $E$ over $\mathbb{A}^1_k$  of the form
	\begin{equation*}
		W:	y^2=x^3+a_2 x^2+a_4x+a_6 
	\end{equation*}
	$a_i \in k[T]$. Then $W$ has discriminant
	\begin{equation}\label{proof_2p1_2}
		\Delta_W=-a_4^3+a_2^2 a_4^2 -a_2^3a_6
	\end{equation}
	and $j$-invariant $j(E)=\frac{a_2^6}{\Delta_W}$. By our assumptions on $E$ we must have $\Delta_W \in k^\times$ as well as $j(E)\notin k$. And $j(E)=\frac{a_2^6}{\Delta_W}$ implies $a_2\notin k$. But by \eqref{proof_2p1_2} $a_4 \in k[T]$ would be a solution of
		\begin{equation*}
			Y^3-Y^2a_2^2+a_2^3a_6+\Delta_W=0
		\end{equation*}
		again contradicting Lemma \ref{proof_23_ldisc} since $a_2\notin k$.\par
	It remains to give an example of an elliptic curve with non-constant $j$-invariant and bad reduction at just two points:
\end{proof}

\begin{example}
\mbox{}
    \begin{enumerate}[leftmargin=*]
      \item		Let $k$ be an algebraically closed field of characteristic $2$. Then the elliptic curve $E/k(T)$ given by
	    \begin{equation*}
		    W: y^2+xy=x^3+x^2+T
	    \end{equation*}	
	    has discriminant $\Delta_W=T$ and hence $W$ is a globally minimal Weierstrass equation for $E$ with bad reduction restricted to $\{0,\infty \}$.
	    The $j$-invariant of $E$ is $1/T$ so $E$ has non-constant $j$-invariant.
      \item		Let $k$ be an algebraically closed field of characteristic $3$. Then the elliptic curve $E/k(T)$ given by
	    \begin{equation*}
		    W: y^2=x^3+Tx^2-T
	    \end{equation*}	
	    has discriminant $\Delta_W=T^4$ and hence $W$ is a globally minimal Weierstrass equation with bad reduction restricted to $\{0,\infty\}$.
	    The $j$-invariant of $E$ is $T^2$ so $E$ has non-constant $j$-invariant.
    \end{enumerate}
\end{example}

Although we do not have in general a two parameter Weierstrass equation $y^2=x^3+Ax+B$ in characteristics $2$ and $3$, we can nevertheless choose simple equations:
\begin{prop}\label{proof_23_weq}
Let $K$ be any field of characteristic $2$ or $3$.
	\begin{enumerate}[leftmargin=*]
		\item 
		 Let $E/K$ be a elliptic curve given by a Weierstrass equation
			\begin{equation*}
				\tilde{W}:	\tilde{y^2}+\tilde{a_1} \tilde{x}\tilde{y}+\tilde{a_3} \tilde{y}=\tilde{x^3}+ \tilde{a_2} \tilde{x^2}+\tilde{a_4} \tilde{x}+\tilde{a_6} 
			\end{equation*}
			Then we can find substitutions to a simpler Weierstrass equation $W$:
					\begin{enumerate}[leftmargin=*]
					
						\item if $\chark{K}=3$ and $j(E)\neq 0$:
							\begin{equation*}
								W: y^2=x^3+a_2 x^2 + a_6
							\end{equation*}
							with $\Delta_W=-a_2^3 a_6$ and $j(E)=-\frac{a_2^3}{a_6}$.
							The only substitutions preserving this form are:
							$x=u^2 x'$, $y=u^3 y'$ with $u\in K^\times$.
							
						\item if $\chark{K}=3$ and $j(E)= 0$:
							\begin{equation*}
								W: y^2=x^3+a_4 x + a_6
							\end{equation*}
							with $\Delta_W=-a_4^3$ and $j(E)=0$.
							The only substitutions preserving this form are:
							$x=u^2 x' + r$, $y=u^3 y'$ with $u\in K^\times, r\in K$.
							
						\item if $\chark{K}=2$ and $j(E)\neq 0$:
							\begin{equation*}
								W: y^2 + xy = x^3+a_2 x^2 + a_6
							\end{equation*}
							with $\Delta_W=a_6$ and $j(E)=\frac{1}{a_6}$.
							The only substitutions preserving this form are:
							$x= x'$, $y=y' + s x'$ with $s\in K$.	
											
						\item if $\chark{K}=2$ and $j(E)= 0$:
							\begin{equation*}
								W: y^2 + a_3 y = x^3 + a_4 x + a_6
							\end{equation*}
							with $\Delta_W=a_3 ^4$ and $j(E)=0$.
							The only substitutions preserving this form are:
							$x= u^2 x' + s^2$, $y=u^3 y' + u^2 s x' + t$ with $u\in K^\times$ and $s,t\in K$.								
					\end{enumerate}		
		\item 
		If we have $K=k(T)$ in (a) with $j(E)\in k$, then we can even choose a globally minimal Weierstrass equation for $E$ over $\al$ of the indicated simple form $W$ as given above.
	\end{enumerate}		
\end{prop}
\begin{proof}
	Part (a) is well known and may be proven by just writing down explicit substitutions. See for example \cite{silv1}[Appendix A, Prop. 1.1 and Proof of Prop. 1.2].\par
	Part (b) is an easy consequence of the proof of (a). Therefore start with a globally minimal Weierstrass equation for $E$ over $\al$:
	\begin{equation*}
		\tilde{W}:	\tilde{y^2}+\tilde{a_1} \tilde{x}\tilde{y}+\tilde{a_3} \tilde{y}=\tilde{x^3}+ \tilde{a_2} \tilde{x^2}+\tilde{a_4} \tilde{x}+\tilde{a_6} 
	\end{equation*}
	with $\tilde{a}_i \in k[T]$
	and check that the explicit substitutions given in [loc. cit.] change the discriminant just by an element in $\ku$ and that we still obtain coefficients in $k[T]$.
\end{proof}

\begin{korollar}\label{proof_3k}
	Let $k$ be an algebraically closed field of characteristic $3$.
	Let $E/k(T)$ be an elliptic curve with $j(E)\in k$ and good reduction everywhere over $\mathbb{A}^1_k$.
	Then $E$ is already constant, if $j(E)\neq 0$.
\end{korollar}
\begin{proof}
	This is an immediate consequence of Proposition \ref{proof_23_weq} (a)(i).
\end{proof}

Part (b) of Theorem \ref{statement_2}, which says that every elliptic curve $E/k(T)$ with good reduction everywhere is constant, may be proven by direct calculations involving Weierstrass equations as in Proposition \ref{proof_23_weq}:

\begin{proof}[Proof of Theorem \ref{statement_2} (b)]
	Let $E$ be an elliptic curve with good reduction everywhere over $\pl$.
	By the proof of Part (a) we have $j(E)\in k$. Let $T$ resp. $T^{-1}$ be uniformizers at $0 \in \pl$ resp. $\infty \in \pl$. We prove the Proposition by choosing globally minimal Weierstrass equations for $E$ over $\pl \setminus \{\infty\}=\spec{k[T]}$ as well as for $\pl \setminus \{0\}=\spec{k[T^{-1}]}$. Then comparing these equations over $\pl \setminus \{\infty,0\}$ yields substitutions giving constant elliptic curves.\par
	  We start with characteristic $3$. By Corollary \ref{proof_3k} we may assume $j(E)=0$ and by Proposition \ref{proof_23_weq} (b) we may choose a globally minimal Weierstrass equation $W$ for $E$ over $\spec{k[T]}=\pl \setminus \{\infty\}$:
	\begin{equation*}
		W:	y^2=x^3+a_4 x+a_6 
	\end{equation*}
	with $a_6 \in k[T]$ and $a_4\in \ku$, since $-a_4^3=\Delta_W\in\ku$.
	We also obtain a globally minimal Weierstrass equation $W'$ for $E$ over $\spec{k[T^{-1}]}=\pl \setminus \{0\}$:
	\begin{equation*}
		W':	y'^2=x'^3+a'_4 x'+a'_6 
	\end{equation*}	
	with $a'_4 \in \ku, a'_6 \in k[T^{-1}]$ say $a'_6=\sum\limits_{i=0}^{k} b_i T^{-i}$. Since $W'$ and $W$ define the same elliptic curve over $k(T)$ there exists a change of variables between $W$ and $W'$.
	By Proposition \ref{proof_23_weq} (a)(ii) such a substitution must be of the form $y=u^3 y'$, $x=u^2 x' + r$ for some $u,r\in k(T)$ . Inserting gives further restrictions: 
	\begin{equation*}\label{proof_2l_eq3}
			u^4=\frac{a_4}{a'_4} \in \ku, \quad r^3+a_4 r -(u^6a'_6 - a_6)=0
	\end{equation*}
	In particular we see $u\in \ku$. Applying Gauss's lemma to $X^3+a_4 X -(u^6a'_6 - a_6)=0$ shows that such a $r\in k(T)$ must already be in $k[T,T^{-1}]$, say $r=\sum\limits_{i=-m}^{n} r_i T^i \in k[T,T^{-1}]$. By setting $r_{+}:=\sum\limits_{i=0}^{n} r_i T^i$ and $r_{-}:=\sum\limits_{i=-m}^{-1} r_i T^i$ we get:
	\begin{equation*}
		\underbrace{(r_{+}^3 + a_4 r_{+}-u^6 b_0 + a_6)}_{\in k[T]} + \underbrace{(r_{-}^3+  a_4 r_{-} -u^6 (a'_6-b_0) )}_{\in T^{-1} k[T^{-1}]}=0
	\end{equation*}
	Thus we see $r_{+}^3 + a_4 r_{+}-u^6 b_0 + a_6=0$, but this means that substituting $y=\bar{y}$, $x= \bar{x} + r_{+}$ in $W$ yields a Weierstrass equation $\bar{W}$ for $E$ with constant coefficients:
	\begin{equation*}
		\bar{W}: \bar{y}^2=\bar{x}^3+a_4 \bar{x}+ u^6 b_0
	\end{equation*}
	with $a_4 \in k$ and $u^6 b_0 \in k$, since $u,b_0 \in k$.\par
	  Since the argument in characteristic $2$ is similar we will just sketch the proof. 
	We first consider the case $j(E)\neq 0$.
	By Proposition \ref{proof_23_weq} (b) we may choose a globally minimal Weierstrass equation $W$ for $E$ over $\spec{k[T]}=\pl \setminus \{\infty\}$ of the form
	\begin{equation*}
		W:	y^2 + xy=x^3+ a_2 x^2+a_6 
	\end{equation*}
	with $a_2 \in k[T], a_6\in \ku$ (we have $a_6=\Delta \in \ku$),\\
	as well as for $E$ over $\spec{k[T^{-1}]}=\pl \setminus \{0\}$:
	\begin{equation*}
		W':	y'^2 + x'y'=x'^3+ a'_2 x'^2+a'_6 
	\end{equation*}	
	with $a'_6 \in \ku, a'_2 \in k[T^{-1}]$ say $a'_2=\sum\limits_{i=0}^{k} b_i T^{-i}$. A substitution between $W$ and $W'$ must be of the form $y=y'+s x'$, $x=x'$ for some $s \in k(T)$ satisfying $s^2+s+a_2+a'_2=0$. Applying Gauss's lemma as above we must have $s=\sum\limits_{i=-m}^{n} s_i T^i \in k[T,T^{-1}]$.
	By setting $s_{+}:=\sum\limits_{i=0}^{n} s_i T^i$ we get $s_{+}^2 + s_{+} + b_0 + a_2=0$, but this means that substituting $y=\bar{y}+s_{+}\bar{x}$, $x= \bar{x}$ in $W$ yields a Weierstrass equation $\bar{W}$ for $E$ with constant coefficients:
	\begin{equation*}
		\bar{W}: \bar{y}^2+\bar{x}\bar{y}=\bar{x}^3+b_0 \bar{x}^2+ a_6
	\end{equation*}
	$a_6 \in \ku$ and $b_0 \in k$.\par
	   Now we consider the case characteristic $2$ and $j(E)=0$.
	Again we get globally minimal Weierstrass equations for $E$ over $\spec{k[T]}$
	\begin{equation*}
		W:	y^2 + a_3 y=x^3+ a_4 x+a_6 
	\end{equation*}
	with $a_3 \in \ku$ and $a_4,a_6 \in k[T]$,
	as well as for $E$ over $\spec{k[T^{-1}]}$
	\begin{equation*}
		W':	y'^2 + a'_3 y'=x'^3+ a'_4 x'+a'_6 
	\end{equation*}
	with $a'_3 \in \ku$ and $a'_4,a'_6 \in k[T^{-1}]$.
	Comparing them over $\spec{k[T,T^{-1}]}$ we get a substitution of the form $x=u^2 x' + s^2$ and $y=u^3 y' + u^2 s x' +t$ with $u,s,t \in k(T)$ (see Proposition \ref{proof_23_weq} (a)(iv)  ) satisfying:
	\begin{align}
		u^3 &= \frac{a_3}{a'_3}  \label{proof_2l_eqsubst_a}\\
		s^4 + a_3 s + a_4 + u^4 a'_4   &= 0  \label{proof_2l_eqsubst_b}\\
		t^2 + a_3 t + s^6 + a_4 s^2 + a_6 + u^6 a'_6 &= 0  \label{proof_2l_eqsubst_c}
	\end{align}
	In particular \eqref{proof_2l_eqsubst_a} yields $u\in \ku$ and together with \eqref{proof_2l_eqsubst_b} we get $s\in \kTf$. After substituting 
	$y=\tilde{y}+s_{+} \tilde{x}$, $x=\tilde{x} + s_{+}^2$ in $W$ as well as $y'=\tilde{y}'+u^{-1}s_{-} \tilde{x}'$, $x'=\tilde{x}' + u^{-2}s_{-}^2$ in $W'$ we may assume $a_4$ and $a_4'$ constant. Then \eqref{proof_2l_eqsubst_c} yields $t\in \kTf$. One more substitution of the form $\tilde{y}=\bar{y}+t_{+}$, $\tilde{x}=\bar{x}$ yields a Weierstrass equation $\bar{W}$ for $E$ with constant coefficients.\par
	The following examples shows that in both characteristics $2$ and $3$ there are non-constant elliptic curves, if we allow bad reduction in one point:
\end{proof}

	\begin{example}\label{proof_2eb}
	\mbox{}
	  \begin{enumerate}[leftmargin=*]
	      \item Let $k$ be an algebraically closed field of characteristic $2$. Then the elliptic curve $E/k(T)$ defined by
		  \begin{equation*}
			  W:\quad y^2+xy=x^3+Tx^2+1
		  \end{equation*}
		  has discriminant $\Delta_W=1$ and $j$-invariant $j(E)=1$.
	  From $\Delta_ {W}\in \ku=k[T]^\times$ we see that the elliptic curve $E$ has good reduction everywhere over $\al$ and a similar computation as done in Example \ref{proof_1eb} shows that $E$ is non-constant.
	      \item Let $k$ be an algebraically closed field of characteristic $3$. Then the elliptic curve $E/k(T)$ given by
		\begin{equation*}
			W: y^2=x^3- x +T
		\end{equation*}
		with $\Delta_W=1$ and $j(E)=0$ has good reduction everywhere over $\al$ and is also
		non-constant.
	  \end{enumerate}

	\end{example}

\bibliographystyle{amsalpha}
\bibliography{ellcurves_p1_arxiv_v2}	

\providecommand{\bysame}{\leavevmode\hbox to3em{\hrulefill}\thinspace}
\providecommand{\MR}{\relax\ifhmode\unskip\space\fi MR }
\providecommand{\MRhref}[2]{%
  \href{http://www.ams.org/mathscinet-getitem?mr=#1}{#2}
}
\providecommand{\href}[2]{#2}
\begin{thebibliography}{Mas84}

\bibitem[Bea81]{bea}
Arnaud Beauville, \emph{Le nombre minimum de fibres singuli\`eres d'une courbe
  stable sur {$\mathbb{P}^1$}.}, Ast\'erisque {\bf 86}, 1981, pp.~97--108.

\bibitem[Fon85]{fon}
Jean-Marc Fontaine, \emph{Il n'y a pas de vari\'et\'e ab\'elienne sur
  {$\mathbb{Z}$}}, Invent. Math. \textbf{81} (1985), no.~3, 515--538.
  \MR{807070 (87g:11073)}

\bibitem[Hus04]{hus}
Dale Husem{\"o}ller, \emph{Elliptic curves}, second ed., Graduate Texts in
  Mathematics, vol. 111, Springer-Verlag, New York, 2004, With appendices by
  Otto Forster, Ruth Lawrence and Stefan Theisen. \MR{2024529 (2005a:11078)}

\bibitem[Liu02]{liu}
Qing Liu, \emph{Algebraic geometry and arithmetic curves}, Oxford Graduate
  Texts in Mathematics, vol.~6, Oxford University Press, Oxford, 2002,
  Translated from the French by Reinie Ern{\'e}, Oxford Science Publications.
  \MR{1917232 (2003g:14001)}

\bibitem[Mas84]{mas}
R.~C. Mason, \emph{Diophantine equations over function fields}, London
  Mathematical Society Lecture Note Series, vol.~96, Cambridge University
  Press, Cambridge, 1984. \MR{754559 (86b:11026)}

\bibitem[Sch05]{schoof}
Ren{\'e} Schoof, \emph{Abelian varieties over {$\mathbb{Q}$} with bad reduction
  in one prime only}, Compos. Math. \textbf{141} (2005), no.~4, 847--868.
  \MR{2148199 (2006c:11068)}

\bibitem[Sil86]{silv1}
Joseph~H. Silverman, \emph{The arithmetic of elliptic curves}, Graduate Texts
  in Mathematics, vol. 106, Springer-Verlag, New York, 1986. \MR{817210
  (87g:11070)}

\end{thebibliography}
		
\end{document}